\documentclass[12pt]{amsart}

\usepackage{amssymb, amsfonts}
\usepackage{url}
\usepackage[all,arc]{xy}
\usepackage{amscd}
\usepackage{mathrsfs}
\usepackage{geometry}
\usepackage[colorlinks,urlcolor=blue]{hyperref}
\usepackage{comment}
\usepackage{enumerate}
\usepackage{graphicx}
\usepackage{bm}
\usepackage{color}
\usepackage{anyfontsize}
\usepackage{caption}

\numberwithin{equation}{section}

\theoremstyle{plain}
\newtheorem{theorem}{Theorem}[section]
\newtheorem{corollary}[theorem]{Corollary}
\newtheorem{lemma}[theorem]{Lemma}
\newtheorem{proposition}[theorem]{Proposition}
\newtheorem{conjecture}[theorem]{Conjecture}

\theoremstyle{remark}
\newtheorem{remark}{Remark}[section]
\newtheorem{example}[remark]{Example}

\begin{document}


\title{On the index of the critical M\"obius band in $\mathbb B^4$.}

\author{Vladimir Medvedev}

\address{Faculty of Mathematics, National Research University Higher School of Economics, 6 Usacheva Street, Moscow, 119048, Russian Federation}

\email{vomedvedev@hse.ru}



\begin{abstract}

In this paper we prove that the Morse index of the critical M\"obius band in the $4-$dimensional Euclidean ball $\mathbb B^4$ equals 5. It is conjectured that this is the only embedded non-orientable free boundary minimal surface of index 5 in $\mathbb B^4$. One of the ingredients in the proof is a comparison theorem between the spectral index of the Steklov problem and the energy index. The latter also enables us to give another proof of the well-known result that the index of the critical catenoid in $\mathbb B^3$ equals 4.

\end{abstract}

\maketitle


\newcommand\cont{\operatorname{cont}}
\newcommand\diff{\operatorname{diff}}

\newcommand{\dvol}{\text{dA}}
\newcommand{\Ric}{\operatorname{Ric}}
\newcommand{\GL}{\operatorname{GL}}
\newcommand{\myO}{\operatorname{O}}
\newcommand{\myP}{\operatorname{P}}
\newcommand{\eye}{\operatorname{Id}}
\newcommand{\myF}{\operatorname{F}}
\newcommand{\Vol}{\operatorname{Vol}}
\newcommand{\odd}{\operatorname{odd}}
\newcommand{\even}{\operatorname{even}}
\newcommand{\ol}{\overline}
\newcommand{\mye}{\operatorname{E}}
\newcommand{\myo}{\operatorname{o}}
\newcommand{\myt}{\operatorname{t}}
\newcommand{\irr}{\operatorname{Irr}}
\newcommand{\mydiv}{\operatorname{div}}
\newcommand{\re}{\operatorname{Re}}
\newcommand{\im}{\operatorname{Im}}
\newcommand{\can}{\operatorname{can}}
\newcommand{\scal}{\operatorname{scal}}
\newcommand{\tr}{\operatorname{trace}}
\newcommand{\sgn}{\operatorname{sgn}}
\newcommand{\SL}{\operatorname{SL}}
\newcommand{\myspan}{\operatorname{span}}
\newcommand{\mydet}{\operatorname{det}}
\newcommand{\SO}{\operatorname{SO}}
\newcommand{\SU}{\operatorname{SU}}
\newcommand{\specl}{\operatorname{spec_{\mathcal{L}}}}
\newcommand{\fix}{\operatorname{Fix}}
\newcommand{\id}{\operatorname{id}}
\newcommand{\grad}{\operatorname{grad}}
\newcommand{\singsup}{\operatorname{singsupp}}
\newcommand{\wave}{\operatorname{wave}}
\newcommand{\ind}{\operatorname{ind}}
\newcommand{\mynull}{\operatorname{null}}
\newcommand{\inj}{\operatorname{inj}}
\newcommand{\arcsinh}{\operatorname{arcsinh}}
\newcommand{\Spec}{\operatorname{Spec}}
\newcommand{\Ind}{\operatorname{Ind}}
\newcommand{\Nul}{\operatorname{Nul}}
\newcommand{\inrad}{\operatorname{inrad}}
\newcommand{\mult}{\operatorname{mult}}
\newcommand{\Length}{\operatorname{Length}}
\newcommand{\Area}{\operatorname{Area}}
\newcommand{\Ker}{\operatorname{Ker}}
\newcommand{\floor}[1]{\left \lfloor #1  \right \rfloor}

\newcommand\restr[2]{{
  \left.\kern-\nulldelimiterspace 
  #1 
  \vphantom{\big|} 
  \right|_{#2} 
  }}


\section{Introduction}

A \textit{free boundary minimal submanifold} $M$ in a Riemannian manifold $(N,g)$ with non-empty boundary is defined as a minimal submanifold whose boundary $\partial M$ lies in $\partial N$ and $M$ meets $\partial N$ orthogonally. The theory of free boundary minimal submaniflods is one of the central topics in geometric analysis. There are numerous results obtained in this direction. Without any hope to list all of them here we refer the interested reader to the survey~\cite{li2019free} and Chapter 1 of the book~\cite{fraser2020geometric}. 

In this paper we study \textit{the (Morse) index} of a free boundary minimal surface in the Euclidean ball. Roughly speaking, the index of a free boundary minimal submanifold is the maximal number of linearly independent infinitesimal variations which decrease the volume of the submanifold up to the second order while its boundary remains in the boundary of the ambient Riemannian manifold. Not much is known about the index of a free boundary minimal submanifold in the Euclidean ball. First of all, it is easy to see that the index of the plane equatorial disk in the unit $n-$dimensional Euclidean ball $\mathbb B^n$ is $n-2$. More generally, the index of an equatorial $\mathbb B^k$ in $\mathbb B^n$ is $n-k$. The first non-trivial results were obtained by Fraser and Schoen in the seminal paper~\cite{fraser2016sharp}. In this paper the authors show that the index of any free boundary minimal surface different from the plane disk in the unit $n-$dimensional Euclidean ball is at least $n$. As a matter of fact, even more general result is obtained: any $k-$dimensional free boundary minimal submanifold in $\mathbb B^n$ under certain assumption has index at least $n$ (see Theorem 3.1 in~\cite{fraser2016sharp}). Later, Sargent in~\cite{sargent2017index} and Ambrosio, Carlotto and Sharp in~\cite{ambrozio2018index} independently gave a lower bound on the index of a free boundary minimal surface in $\mathbb B^3$  in terms of the genus and the number of boundary components. Note that this estimate also works in a more general setting of mean convex domains in $\mathbb R^3$. In the case of free boundary minimal hypersurfaces in $\mathbb B^n$  it was shown in~\cite{devyver2019index} that the index is at least $n+1$. Also in~\cite[Theorems A and B]{ambrozio2018index} lower bounds on the index of a free boundary minimal hypersurface in a strictly mean convex domain in $\mathbb R^n$ were obtained.  The asymptotic of the index of $n-$dimensional critical catenoids in the unit $(n+1)-$dimensional Euclidean balls as $n\to \infty$ is studied by Smith, Stern, Tran and Zhou in the paper~\cite{smith2017morse}. 

The case of higher codimension is more complicated. To the best of our knowledge there is only one result on the index of a free boundary minimal surface in $\mathbb B^n, n\geqslant 4$. It is the result of Lima in~\cite{lima2017bounds}. In this paper the author obtains a comparison theorem between the index of a free boundary minimal surface in a general Riemannian manifold with boundary and \textit{the energy index}. Roughly speaking, the energy index is the maximal number of linearly independent infinitesimal variations  which decrease the energy of an immersed or embedded minmal submanifold up to the second order. In the case of free boundary minimal submanifolds we additionally require that the boundary of the submanifold does not leave the boundary of the ambient Riemannian manifold. In~\cite{lima2017bounds} the author also obtains inequalities involving \textit{the nullity} of a free boundary minimal submanifold which is defined as the maximal number of linearly independent infinitesimal variations on which the second variation of the volume functional vanishes.

In the aforementioned paper~\cite{fraser2016sharp} by Fraser and Schoen the authors study two important examples of free boundary minimal surfaces in the Euclidean balls: \textit{the critical catenoid} in $\mathbb B^3$ and \textit{the critical M\"obius band} in $\mathbb B^4$. Later, Devyver~\cite{devyver2019index}, Tran~\cite{tran2016index} and Smith and Zhou~\cite{smith2019morse} independently computed the index of the critical catenoid in $\mathbb B^3$. 

\begin{theorem}[Devyver~\cite{devyver2019index}, Tran~\cite{tran2016index}, Smith-Zhou~\cite{smith2019morse}]\label{cat}
The index of the critical catenoid in the $3-$dimensional Euclidean ball equals 4.
\end{theorem}

Devyver in~\cite{devyver2019index} also proved that the index of any free boundary minimal surface different from the plane disk in $\mathbb B^3$ has index at least 4 which improves the estimate of Fraser and Schoen in~\cite{fraser2011first}. It is conjectured that the critical catenoid is the only embedded free boundary minimal surface in $\mathbb B^3$ of index 4. This conjecture was partially proved in~\cite{tran2016index, devyver2019index}. Note also that the critical catenoid is conjectured to be the only free boundary minimal annulus in $\mathbb B^3$. This was partially proved in~\cite{mcgrath2016characterization,kusner2020free}. 

In this paper we compute the index of the critical M\"obius band.

\begin{theorem}\label{mobius_intro}
The index of the critical M\"obius band in the $4-$dimensional Euclidean ball equals 5. 
\end{theorem}

The main difficulty in this computation is that the critical M\"obius band in $\mathbb B^4$ has codimension 2. In the above results mainly the case of the codimension one was considered. In this case the problem of the index estimate can be reduced to the eigenvalue problem of the stability operator on functions. In contrast to this, one has to deal directly with normal vector fields in order to estimate the index of a free boundary minimal submanifold of higher codimension. While in general this seems quite difficult to realize, the case of surfaces looks a little bit simpler since the methods of complex geometry can be used. This is what we do in order to compute the index of the critical M\"obius band. Our strategy is as follows. In order to get the lower bound on the index of the critical M\"obius band we pass to its orientable double cover. On this cover we can find five linearly independent normal vector fields which contribute to the index (see Theorem~\ref{main_intro}). In order to find these fields we use the approach of Kusner and Wang in the paper~\cite{kusner2018index}. This result is an analog of~\cite[Theorem 3.1 (1)]{kusner2018index} for the case of free boundary minimal surfaces in $\mathbb B^n$. Note that the application of complex geometry and the Hopf differentials to the theory of free boundary minimal surfaces was initiated in the paper~\cite{nitsche1985stationary} and developed in the papers~\cite{fraser2007index, fraser2015uniqueness} (see also Chapter 1 of the book~\cite{fraser2020geometric}). Further, it turns out that the found five fields descend to the critical M\"obius band, which shows that the index of the critical M\"obius band is at least 5. The upper bound is a corollary of a comparison theorem between the spectral index of the Steklov problem and the energy index (see Theorem~\ref{indexES} below). This theorem implies that the index of the critical catenoid is at most 5. Also this theorem enables us to give another proof of Theorem~\ref{cat}.




By analogy with the critical catenoid one can formulate the following conjecture

\begin{conjecture}
The critical M\"obius band is the only embedded non-orientable free boundary minimal surface in the $4-$dimensional Euclidean ball of index 5.
\end{conjecture}

\subsection{Discussion}

It is well known that the theory of closed minimal submanifolds in the standard sphere is closely related to the geometric optimization of eigenvalues of \textit{the Laplace-Beltrami operator} (see for example the surveys~\cite{penskoi2013extremal, penskoi}). In the same spirit the theory of free boundary minimal submanifolds in the unit Euclidean ball is related to the geometric optimization of eigenvalues of \textit{the Steklov problem} as it was first discovered by Fraser and Schoen in the papers~\cite{fraser2011first, fraser2016sharp}. In order to define this problem we will assume that $(M,g)$ is a Riemannian manifold with non-empty Lipshitz boundary.  Then the Steklov problem is the following eigenvalue problem
$$
\begin{cases}
\Delta_g u=0~\text{in $M$},\\
\partial_\eta u=\sigma u~\text{on $\partial M$},
\end{cases}
$$ 
where $u\in C^\infty(M), ~\Delta_g$ is the Laplace-Beltrami operator of the metric $g$ and $\eta$ is the outward unit normal field to the boundary. The real numbers $\sigma$ such that the Steklov problem admits non-trivial solutions are called \textit{Steklov eigenvalues}. The corresponding solutions $u$ are called \textit{Steklov eigenfunctions}.  We refer the interested reader to the survey~\cite{girouard2017spectral} for more information about the Steklov problem. Here we mention that any free boundary minimal submanifold in $\mathbb B^n$ is given by Steklov eigenfunctions with eigenvalue 1.  

Recently, Karpukhin and Metras in~\cite{karpukhin2021laplace} studied the $n-$harmonic maps and introduced the notion of the spectral index of a Riemannian manifold with boundary as the number of Steklov eigenvalues not exceeding 1. Previously, the spectral index of a closed Riemannian surface was introduced in~\cite{montiel1991schrodinger} and studied in~\cite{karpukhin2021index}. In the latter paper a comparison theorem between the spectral index and the energy index was obtained. If we denote the spectral index of a surface $\Sigma$ as $\Ind_S(\Sigma)$ and the energy index as $\Ind_E(\Sigma)$ then the following theorem holds

\begin{theorem}\label{indexES}
Let $u\colon \Sigma \to \mathbb B^n$ be a free boundary minimal immersion of a surface $\Sigma$ into $n-$dimensional Euclidean ball. Then
$$
\Ind_E(\Sigma)\leqslant n\Ind_S(\Sigma).
$$
\end{theorem}

This theorem is the second ingredient in the proof of Theorem~\ref{mobius_intro}. However, we believe that Theorem~\ref{indexES} could be of independent interest. 

We finish the discussion with the following two theorems. The first one was inspired by~\cite[Lemma 7.1]{devyver2019index},~\cite[Theorem 3.8]{tran2016index} (see also~\cite{fraser2020geometric})

\begin{theorem}\label{+n}
Let $\Sigma$ be a non-flat free boundary minimal hypersurface in $\mathbb B^n$. Then one has
$$
\Ind(\Sigma) \geqslant \Ind_S(\Sigma)+n.
$$
\end{theorem}

The second theorem was first proved in the paper~\cite{lima2017bounds} (see Theorem 1 therein) for the case of orientable free boundary minimal surfaces. The proof provided below in Section~\ref{appendix} is based on the original ideas by Fraser and Schoen in the paper~\cite{fraser2016sharp} (see Propositions 6.5 and 7.3 therein) and works for non-orientable  free boundary minimal surfaces as well as for orientable ones.
 
\begin{theorem}\label{lima}
Let $\Sigma$ be a (orientable or non-orientable) free boundary minimal surface in $\mathbb B^n$. Then
$$
\Ind(\Sigma) \leqslant \Ind_E(\Sigma)+\dim \mathcal M(\Sigma).
$$
where $\mathcal M(\Sigma)$ is the moduli space of conformal structures on $\Sigma$.
\end{theorem}



With Theorem~\ref{lima} Theorems~\ref{+n} and~\ref{indexES} imply a two-sided inequality on the energy and spectral indices. A similar two-sided inequality on the energy and spectral indices of a closed minimal surface in $\mathbb S^n$ was successfully used in~\cite{karpukhin2021index}. It would be interesting to obtain similar results for higher dimensional free boundary minimal submanifolds in $\mathbb B^n$.

\subsection{Plan of the paper} The paper is organized in the following way. Section~\ref{notations} contains the notation and definitions that we use throughout the paper. In Section~\ref{preliminaries} we recall some facts about free boundary minimal surfaces in the Euclidean balls. Section~\ref{tech} contains a technical background useful for the consequent sections. In Section~\ref{proof} we prove some auxiliary theorem (Theorem~\ref{main_intro}) which will later enable us to estimate the index of the critical catenoid in $\mathbb B^4$ from below. Here we also consider the case of Fraser-Sargent surfaces (see Theorem~\ref{indFS}). In Section~\ref{ES} we give the proofs of Theorems~\ref{indexES} and~\ref{+n} and deduce Corollary~\ref{cor} which we use in the following section. Section~\ref{mobius} contains the proof of Theorem~\ref{mobius_intro} and in Subsection~\ref{Cat} we give another proof of Theorem~\ref{cat}. Finally, in Section~\ref{appendix} we prove Theorem~\ref{lima}.

\subsection*{Acknowledgements} The author is deeply indebted to Mikhail Karpukhin for bringing his attention to the theory of free boundary minimal submanifolds as well as for numerous and fruitful discussions. The author is also grateful to Misha Verbitsky and Alexei Penskoi for useful discussions. The author would like to thank Iosif Polterovich for valuable remarks on the preliminary version of the manuscript. The author is also grateful to the reviewer for valuable remarks and suggestions. During the work on the paper the author was partially supported by the Simons-IUM fellowship, by the contest "Young Russian Mathematics" and by the Theoretical Physics and Mathematics Advancement Foundation "BASIS". 

\section{Notation and definitions} \label{notations}

Throughout the paper we use the following notation and definitions.  
 
\begin{itemize}
\item $\mathbb B^n$ is the unit ball centred at the origin in the Euclidean space $\mathbb R^n$; 
\item $\Sigma$ is a free boundary minimal surface in $\mathbb B^n$ given by the free boundary immersion $u\colon \Sigma \to \mathbb B^n$; 
\item $-\cdot-$ denotes the standard Euclidean dot-product;
\item $\langle-,-\rangle$ and $g$ denote the scalar product and the metric induced on $\Sigma$;
\item $\Gamma(N\Sigma), \Gamma(T\Sigma)$ denote the sections of the normal bundle $N\Sigma$ and the tangent bundle $T\Sigma$ over $\Sigma$ respectively;
\item for any vector $v\in \mathbb R^n$ $v^\perp$ denotes the projection onto $\Gamma(N\Sigma)$ and $v^\top$ is the projection onto $\Gamma(T\Sigma)$;
\item $\nabla^\perp$ is the connection in $N\Sigma$ and $\nabla^\top$ is the connection in $T\Sigma$; the covariant derivative on $\mathbb R^n$ is denoted by $\nabla$;
\item the Laplacian on the normal bundle is defined by 
$$
\Delta^\perp X=\sum_{i=1}^2\left(\nabla^\perp_{e_i}\nabla^\perp_{e_i} X-\nabla^\perp_{(\nabla_{e_i}e_i)^\top}X\right),~\forall X\in \Gamma(N\Sigma),
$$
here $e_1,e_2$ is a local orthonormal basis in $\Gamma(T\Sigma)$; 
\item the second fundamental form of $\Sigma$ is given by $B(X,Y)=(\nabla_XY)^\perp, \forall X,Y \in \Gamma(T\Sigma)$, particularly, $b_{ij}=B(e_i,e_j)$.
\item the Simons operator on $X\in \Gamma(N\Sigma)$ is defined as $\mathcal B(X)=\sum_{i,j=1}^2(b_{ij}\cdot X)b_{ij}$;
\item the Jacobi operator on $X\in \Gamma(N\Sigma)$ is given by the formula 
$$
L(X)=\Delta^\perp X+\mathcal B(X);
$$
\item the second variation of the area of $\Sigma$ towards the direction $X\in \Gamma(N\Sigma)$ is the following quadratic form:
$$
S(X,X)=-\int_\Sigma \langle L(X), X\rangle dv_g+\int_{\partial \Sigma} \left(\langle X,\nabla^\perp_\eta X\rangle-|X|^2\right)ds_g,
$$
here $|X|^2=\langle X,X\rangle$ and $\eta$ is the outward unit normal field to the boundary;
\item the (Morse) index $\Ind(\Sigma)$ is the maximal dimension of a vector subspace $V\subset \Gamma(N\Sigma)$ on which $S$ is negative-definite;
\item $\Nul(\Sigma)$ is the nullity of $\Sigma$ which is defined as the maximal dimension of a vector subspace $V\subset \Gamma(N\Sigma)$ on which $S$ vanishes;
\item the second variation of the energy of $\Sigma$ towards the direction $X\in \mathbb R^n$ is the following quadratic form:
$$
S_E(X,X)=\int_\Sigma |\nabla X|^2 dv_g-\int_{\partial \Sigma} |\nabla_\eta u||X|^2ds_g;
$$
\item the energy (Morse) index $\Ind_E(\Sigma)$ is the maximal dimension of a vector subspace $V\subset \Gamma(T\mathbb R^n)$ on which $S_E$ is negative-definite; notice that in the problem of the energy index estimates from below it suffices to consider harmonic vector fields since they have least energy; 
\item the spectral index is defined as the number of negative eigenvalues of the following operator
$$
L^S(\varphi)=\eta \hat\varphi-|\nabla_\eta u|\varphi,~\forall \varphi \in C^\infty(\partial \Sigma),
$$
where $\hat\varphi$ denotes the harmonic continuation of $\varphi$ (for details see~\cite{karpukhin2021laplace}); the corresponding quadratic form is denoted as
$$
S_S(\varphi,\varphi)=\int_\Sigma|\nabla \hat\varphi|^2dv_g-\int_{\partial \Sigma} |\nabla_\eta u|\varphi^2ds_g;
$$
\item $\mathbb K$ is the critical catenoid
which is defined as the image of the following free boundary minimal map:
$$
u\colon [-T_K,T_K]\times \mathbb S^1 \to \mathbb B^3, 
$$ 
where $u(t,\theta)=\frac{1}{r}(\cosh t\cos\theta, \cosh t\sin \theta,t)$, $T_K$ is the unique positive solution of the equation $\coth t=t$ and $r=\sqrt{\cosh^2T_K+T^2_K}$ (see~\cite{fraser2016sharp});
\item $\mathbb M$ is the critical M\"obius band which is defined as the image of the following free boundary minimal map:
$$
u\colon [-T_M,T_M]\times \mathbb S^1/\sim \to \mathbb B^4, 
$$ 
where $u(t,\theta)=(2\sinh t\cos\theta, 2\sinh t\sin \theta, \cosh 2t\cos 2\theta, \cosh 2t\sin 2\theta)$ and $T_M$ is the unique positive solution of the equation $\coth t=2\tanh 2t$, $\sim$ is the following equivalence relation $u(t,\theta)\sim u(-t,\theta+\pi)$ (see~\cite{fraser2016sharp});
\item $\mathbb K_q$ is the critical $q-$catenoid which is defined as the image of the following free boundary minimal immersion:
$$
u\colon [-t_{1,0}/q,t_{1,0}/q]\times \mathbb S^1 \to \mathbb B^3, 
$$ 
where $u(t,\theta)=\frac{1}{r_q}(\cosh (qt)\cos(q\theta), \cosh (qt)\sin (q\theta),qt)$, $q$ is a natural number, $r_q=\sqrt{\cosh^2(qt_{1,0})+t^2_{1,0}}$ and $t_{1,0}$ is the unique positive solution of $\coth t=t$ (see~\cite{fraser2016sharp,fraser2021existence}). Note that $\mathbb K_1=\mathbb K$;
\item the Fraser-Sargent annuli (see~\cite{fan2015extremal,fraser2021existence}) in $\mathbb B^4$ are defined as 
$$
u\colon [-t_{k,l},t_{k,l}]\times \mathbb S^1 \to \mathbb B^4, 
$$
where
\begin{gather*}
u(t,\theta)=\\ \frac{1}{r_{k,l}}(k\sinh(lt)\cos(l\theta),k\sinh(lt)\sin(l\theta),l\cosh(kt)\cos(k\theta),l\cosh(kt)\sin(k\theta)),
\end{gather*}
$k,l \in \mathbb N$ with $k>l$, $r_{k,l}=\sqrt{k^2\sinh^2(lt_{k,l})+l^2\cosh^2(kt_{k,l})}$ and $t_{k,l}$ is the unique positive solution of $k\tanh(kt)=l\coth(lt)$. 
\end{itemize}

\section{Preliminaries}\label{preliminaries}

In this section we collect some known facts about free boundary minimal surfaces in the Euclidean ball which we use in the subsequent sections.

\begin{theorem}[Fraser-Schoen~\cite{fraser2016sharp}]\label{FS}
Let $v\in \mathbb R^n\setminus \{0\}$. Then for the second variation of the area of $\Sigma$ towards $v^\perp$ one has
$$
S(v^\perp,v^\perp)=-2\int_\Sigma |v^\perp|^2dv_g.
$$
Moreover, if $\Sigma$ is not a plane disk and  $v_1\perp v_2$ then $S(v_1^\perp,v_2^\perp)=0$. Particularly, if $\Sigma$ is not a plane disk then $\Ind(\Sigma) \geqslant n$.
\end{theorem}

The following proposition is commonly known.

\begin{proposition}\label{vperp}
The normal field $v^\perp$ on $\Sigma$ is a Jacobi field, i.e. it satisfies the equation $L(v^\perp)=0$, where $L$ is the Jacobi operator.
\end{proposition}

As we mention in the Introduction there exists an explicit lower bound on the index of a free boundary minimal hypersurface.

\begin{theorem}[Devyver~\cite{devyver2019index}]\label{devyver}
Let $\Sigma$ be a non-flat free boundary minimal hypersurface in $\mathbb B^n$. Then $\Ind(\Sigma)\geqslant n+1$.
\end{theorem}

\section{Technical results}\label{tech}

We will use the approach described in~\cite{kusner2018index} (see also~\cite[Section 6]{karpukhin2021stability} for the non-orientable case). 

\medskip

Choose isothermal local coordinates $(x,y)$ on $\Sigma$. Then the metric $g$ on $\Sigma$ takes the form $g=e^{2\omega}(dx^2+dy^2)$, where $\omega \in C^\infty(\Sigma)$. Further, introduce the local complex coordinate $z=x+iy$. Then $g=e^{2\omega}|dz|^2$. Let $E$ be either the normal or the tangent bundle. For any local sections $X, Y$ of $E \otimes_{\mathbb R}\mathbb C$ we also use the Hermitian scalar product $X\cdot \bar Y$, where $\bar Y$ is conjugate to $Y$. Particularly, $|X|^2=X\cdot \bar X$. In the coordinates $(x,y)$ and $z$ the immersion $u\colon \Sigma \to \mathbb B^n$ is conformal and harmonic. Hence the following claim is obvious

\medskip

{\bf Claim 1.} One has

\begin{itemize}
\item $|u_x|^2=u_x\cdot u_x=u_y\cdot u_y=|u_y|^2=e^{2\omega}$ and $|u_z|^2=|u_{\bar z}|^2=u_z\cdot u_{\bar z}=\frac{1}{2}e^{2\omega}$;
\item $u_z\cdot u_z=u_{\bar z}\cdot u_{\bar z}=0$ and $u_{z\bar z}=0$.
\end{itemize}

\medskip

Using the notation in~\cite{kusner2018index} we set $u_{zz}^\perp=\Omega$. Note that $\Omega$ is a local section of $N\Sigma\otimes_{\mathbb R}\mathbb C$. We also use the notation $\nabla^\perp_z:=\nabla^\perp_{\partial/\partial z}$ and $\nabla^\perp_{\bar z}:=\nabla^\perp_{\partial/\partial \bar z}$ and the similar notation for $\nabla^\top_z$ and $\nabla^\top_{\bar z}$.

\medskip

{\bf Claim 2.} One has
$$
\begin{cases}
u_{zz}=2\omega_zu_z+\Omega,\\
X_z=\nabla^\perp_zX-2e^{-2\omega}(X\cdot\Omega)u_{\bar z},\\ 
X_{\bar z}=\nabla^\perp_{\bar z}X-2e^{-2\omega}(X\cdot \bar\Omega)u_z, 
\end{cases}
$$
for any local section $X$ of $N\Sigma\otimes_{\mathbb R}\mathbb C$.

\medskip

\begin{proof}
By Claim 1 $u_z$ and $u_{\bar z}$ are perpendicular with respect to the Hermitian scalar product. Then the projection formula implies
$$
u_{zz}=\frac{u_{zz}\cdot u_{\bar z}}{|u_z|^2}u_z+\frac{u_{zz}\cdot u_z}{|u_{\bar z}|^2}u_{\bar z}+\Omega.
$$ 
By Claim 1
\begin{gather*}
u_z\cdot u_z=0,\\
u_z\cdot u_{\bar z}=\frac{1}{2}e^{2\omega},
\end{gather*}
which implies
\begin{gather*}
u_{zz}\cdot u_z=0,\\
u_{zz}\cdot u_{\bar z}=\omega_ze^{2\omega}
\end{gather*}
Substituting it in the projection formula and using Claim 1 once again we get the first identity.

Similarly, to get the second identity we use the projection formula
$$
X_z=(X_z)^\perp+\frac{X_z\cdot u_{\bar z}}{|u_z|^2}u_z+\frac{X_z\cdot u_z}{|u_{\bar z}|^2}u_{\bar z}.
$$
Note that by definition $(X_z)^\perp=\nabla_z^\perp X$ and
$$
X\cdot u_z=0=X\cdot u_{\bar z},
$$
for any local section $X$ of $N\Sigma\otimes_{\mathbb R}\mathbb C$, whence
\begin{gather*}
X_z\cdot u_z=-X\cdot u_{zz}=-X\cdot (u_{zz})^\perp ,\\
X_z\cdot u_{\bar z}=-X\cdot u_{z\bar z}=0
\end{gather*}
by Claim 1. Using the formula for $u_{zz}$ and Claim 1 once again completes the proof of the second identity. The proof of the third identity is absolutely similar.
\end{proof}

{\bf Claim 3.} The following identities hold
$$
\begin{cases}
\nabla^\perp_{\bar z}\Omega=0,\\
\nabla^\perp_{\bar z}\nabla^\perp_z X-\nabla^\perp_z\nabla^\perp_{\bar z}X=2e^{-2\omega}\left((X\cdot\Omega)\bar\Omega-(X\cdot\bar\Omega)\Omega\right), 
\end{cases}
$$
for any local section $X$ of $N\Sigma\otimes_{\mathbb R}\mathbb C$.

\begin{proof}
By definition $\nabla^\perp_{\bar z}\Omega=(\Omega_{\bar z})^\perp=\left((u_{zz})^\perp_{\bar z}\right)^\perp$. The projection formula yields
$$
(u_{zz})^\perp=u_{zz}-\frac{u_{zz}\cdot u_{\bar z}}{|u_z|^2}u_z-\frac{u_{zz}\cdot u_z}{|u_{\bar z}|^2}u_{\bar z}=u_{zz}-\frac{u_{zz}\cdot u_{\bar z}}{|u_z|^2}u_z,
$$
since $u_{zz}\cdot u_z=0$ as we have seen in the proof of Claim 2. Differentiaiting implies:
$$
(u_{zz})^\perp_{\bar z}=u_{zz\bar z}-\left(\frac{u_{zz}\cdot u_{\bar z}}{|u_z|^2}\right)_{\bar z}u_z-\frac{u_{zz}\cdot u_{\bar z}}{|u_z|^2}u_{z\bar z}=-\left(\frac{u_{zz}\cdot u_{\bar z}}{|u_z|^2}\right)_{\bar z}u_z,
$$ 
since $u_{z\bar z}=0$ by Claim 1. Hence, $\left((u_{zz})^\perp_{\bar z}\right)^\perp=0$.

Let's prove the second identity. For any local section $X$ of $N\Sigma\otimes_{\mathbb R}\mathbb C$ by Claim 2 one has
\begin{gather*}
X_{\bar zz}=\nabla^\perp_{\bar z} X_z-2e^{-2\omega}(X_z\cdot \bar\Omega)u_z=\\=\nabla^\perp_{\bar z}\nabla^\perp_zX-\left(2e^{2\omega}(X\cdot \Omega)\right)_{\bar z}u_{\bar z}-2e^{-2\omega}(X\cdot \Omega)(u_{\bar z\bar z})^\perp-2e^{-2\omega}(X_z\cdot \bar\Omega)u_z
\end{gather*}
and
\begin{gather*}
X_{z\bar z}=\nabla^\perp_z X_{\bar z}-2e^{-2\omega}(X_{\bar z}\cdot \Omega)u_{\bar z}=\\=\nabla^\perp_z\nabla^\perp_{\bar z}X-\left(2e^{2\omega}(X\cdot \bar\Omega)\right)_zu_z-2e^{-2\omega}(X\cdot \bar\Omega)(u_{zz})^\perp-2e^{-2\omega}(X_{\bar z}\cdot \Omega)u_{\bar z}.
\end{gather*}
Then the second identity in the claim follows from the fact that
$$
(X_{\bar zz})^\perp=(X_{z\bar z})^\perp.
$$
\end{proof}

{\bf Claim 4.} The Laplacian on the normal bundle takes the form
$$
\Delta^\perp X=2e^{-2\omega}\left(\nabla^\perp_{\bar z}\nabla^\perp_z X+\nabla^\perp_z\nabla^\perp_{\bar z}X\right), 
$$
for any local section $X$ of $N\Sigma\otimes_{\mathbb R}\mathbb C$.

\begin{proof}
This formula immediately follows from the formula for the Laplacian on the normal bundle in Section~\ref{notations}.
\end{proof}

Since $\Omega$ is a local section of $N\Sigma\otimes_{\mathbb R}\mathbb C$ then we introduce the local sections $\Omega_1,\Omega_2$ of $N\Sigma$ such that $\Omega=\Omega_1+i\Omega_2$. We now specialize to the case of the annulus. Let $\Sigma=[-T,T]\times \mathbb S^1$ be a free boundary minimal annulus for some $T>0$ and $z=t+i\theta, (t,\theta)\in [-T,T]\times \mathbb S^1$ be the complex coordinate. Note that this coordinate is global on $\Sigma$. Also suppose that $\Sigma$ is $\mathbb S^1$-symmetric. Hence $|u_t|=|u_\theta|=e^\omega=const$ along $\partial\Sigma$. It follows from~\cite[Theorem 1.3]{fraser2021existence} that $\Sigma$ is either a Fraser-Sargent surface or a critical $q-$catenoid.

\medskip

{\bf Claim 5.} One has
\begin{itemize}
\item $\Omega_1=\frac{1}{2}e^{2\omega}b_{11}$ and $\Omega_2=-\frac{1}{2}e^{2\omega}b_{12}$.
\item $\Omega_2=0$ along the boundary.
\end{itemize}

\begin{remark}
The first item in Claim 5 holds for any minimal surface.
\end{remark}

\begin{proof}
By definition one has $B(\partial/\partial x_i,\partial/\partial x_j)=\nabla^\perp_{\partial/\partial x_i}\frac{\partial u}{\partial x_j}$, where $x_1=t$ and $x_2=\theta$. A straightforward computation shows
\begin{gather*}
\Omega=(u_{zz})^\perp=\nabla_z^\perp u_z=\frac{1}{4}\left(B(\partial/\partial t,\partial/\partial t)-B(\partial/\partial \theta,\partial/\partial \theta)-2iB(\partial/\partial t,\partial/\partial \theta)\right)=\\=\frac{1}{4}e^{2\omega}\left(b_{11}-b_{22}-2ib_{12}\right).
\end{gather*}
By the minimality of $\Sigma$ one has $b_{11}+b_{22}=0$, which implies $\Omega=\frac{1}{2}e^{2\omega}\left(b_{11}-ib_{12}\right)$. The first item is proved.

Let us prove that $b_{12}=0$ along the boundary. Let $p\in\partial \Sigma$. Then $e^{-\omega}\frac{\partial u}{\partial t}=\eta$ is the outward unit normal and $e^{-\omega}\frac{\partial u}{\partial \theta}=\tau$ is a unit tangent to $\partial \Sigma$. One has
$$
b_{12}(p)=\nabla^\perp_\tau\eta=\tau^\perp=0,
$$
since $\eta$ is the position vector along the boundary. Since the point $p$ was chosen arbitrarily, we get that $b_{12}=0$ along the boundary.
\end{proof}

As we have already noticed the coordinate $z$ is global on $\Sigma$. Therefore, the field $\partial/\partial z$ is globally defined on $\Sigma$ as well as the fields $u_{zz}, u_z$ and the function $\omega_z$. Hence by Claim 2 so is the normal vector field $\Omega$. We then introduce \textit{the quartic Hopf differential} $\mathcal H=(\Omega\cdot\Omega)dz^4$. In the proof of the following proposition we show that $\mathcal H$ is a holomorphic quartic differential which is real on the boundary of $\Sigma$.  

\begin{proposition}\label{dot}
The function $\Omega\cdot\Omega$ is a real constant.
\end{proposition}

\begin{proof}
Essentially, the proof is given in~\cite[Section 1.5.2]{fraser2020geometric}. For the sake of completeness we give it here.    

First, we prove that $(\Omega\cdot\Omega)_{\bar z}=0$. Indeed,
$$
(\Omega\cdot\Omega)_{\bar z}=2\Omega_{\bar z}\cdot\Omega=2(\Omega_{\bar z})^\perp\cdot\Omega=2\nabla^\perp_{\bar z}\Omega \cdot \Omega=0
$$
by Claim 3. Hence, $\Omega\cdot\Omega$ is holomorphic and $\mathcal H$ is a holomorphic quartic differential.

Further, consider the field $\partial/\partial \theta$. One has $dz(\partial/\partial \theta)=(dt+id\theta)(\partial/\partial \theta)=i$ whence 
$$
\mathcal H(\partial/\partial \theta)=\Omega\cdot\Omega =\frac{1}{4}e^{4\omega}|b_{11}|^2
$$
 on $\partial\Sigma$ since $b_{12}=0$ by Claim 5. Hence, the function $\mathcal H(\partial/\partial \theta)$ is holomorphic on $\Sigma$ and real on $\partial \Sigma$. Thus, $\mathcal H(\partial/\partial \theta)=\Omega\cdot\Omega$ is a real constant. 
\end{proof}

{\bf Claim 6.}  One has
\begin{itemize}
\item $\Omega_1\cdot \Omega_2=0$ and $\Omega\cdot\Omega=|\Omega_1|^2-|\Omega_2|^2$;
\item $\Delta^\perp\Omega=4e^{-4\omega}\left((\Omega\cdot\Omega)\bar\Omega-(\Omega\cdot\bar \Omega)\Omega\right)$;
\item $\Delta^\perp\Omega_1=-8e^{-4\omega}|\Omega_2|^2\Omega_1$ and $\Delta^\perp\Omega_2=-8e^{-4\omega}|\Omega_1|^2\Omega_2$;
\item $\mathcal B(\Omega_j)=8e^{-4\omega}|\Omega_j|^2\Omega_j, j=1,2$.
\end{itemize}

\begin{proof}
Let's prove the first item. One has
$$
\Omega\cdot\Omega=|\Omega_1|^2-|\Omega_2|^2-2i\Omega_1\cdot\Omega_2.
$$
Since by Proposition~\ref{dot} $\Omega\cdot\Omega$ is real we get that $\Omega_1\cdot\Omega_2=0$ and $\Omega\cdot\Omega=|\Omega_1|^2-|\Omega_2|^2$ by comparing the real and imaginary parts.

In order to get the second item we apply the formula for the Laplacian in Claim 4 and then we use Claim 3.

The third item follows from the formula
\begin{gather*}
\Delta^\perp \Omega=\Delta^\perp\Omega_1+i\Delta^\perp\Omega_2=\\4e^{-4\omega}\left(((\Omega_1+i\Omega_2)\cdot(\Omega_1+i\Omega_2))(\Omega_1-i\Omega_2)-((\Omega_1+i\Omega_2)\cdot (\Omega_1-i\Omega_2))(\Omega_1+i\Omega_2))\right)
\end{gather*}
by comparing the real and imaginary parts.    

Finally, in order to prove the last item we use the explicit formula for the Simons operator in Section~\ref{notations}.
\end{proof}

\section{An auxiliary theorem}\label{proof} 

Our aim in this section is to prove the following theorem 

\begin{theorem}\label{main_intro}
Let $\Sigma$ be a Fraser-Sargent annulus in $\mathbb B^4$ or a critical $q-$catenoid in $\mathbb B^3$.  Let $v_1,\ldots,v_n$ be the standard basis of $\mathbb R^n$, where $n=4$ or $3$ respectively. Then the second variation of the volume functional $S$ is negative definite on $span\{\Omega_1,v_1^\perp,\ldots,v_n^\perp\}$ and the fields $\Omega_1,v_1^\perp,\ldots,v_n^\perp$ are linearly independent.  Particularly,  the index of $\Sigma$ is at least $n+1$. 
\end{theorem}

Before proving Theorem~\ref{main_intro} we provide some computational examples illustrating that the quartic Hopf differentials of Fraser-Sargent surfaces, the critical M\"obius band and critical $q-$catenoids do not vanish.

\begin{example}\label{ExFS}
Let us show that $\mathcal H\neq 0$ for Fraser-Sargent surfaces. Recall that they are given by the following formula:
$$
u(t,\theta)=\frac{1}{r_{k,l}}(k\sinh(lt)\cos(l\theta),k\sinh(lt)\sin(l\theta),l\cosh(kt)\cos(k\theta),l\cosh(kt)\sin(k\theta)),
$$
where $k,l \in \mathbb N$ with $k>l$, $r_{k,l}=\sqrt{k^2\sinh^2(lt_{k,l})+l^2\cosh^2(kt_{k,l})}$ and $t_{k,l}$ is the unique positive solution of $k\tanh(kt)=l\tanh(lt)$.

By Claim 5 $\Omega_1=\frac{1}{2}e^{2\omega}b_{11}=\frac{1}{2}B(\partial/\partial t,\partial/\partial t)=\frac{1}{2}u_{tt}^\perp$ and $\Omega_2=-\frac{1}{2}e^{2\omega}b_{12}=\frac{1}{2}B(\partial/\partial t,\partial/\partial \theta)=\frac{1}{2}u_{t\theta}^\perp$. Thus, in order to show that $\mathcal H\neq 0$ by Proposition~\ref{dot} and the first item of Claim 6 it suffices to prove that $|u_{tt}^\perp|^2-|u_{t\theta}^\perp|^2 \neq 0$.
We find that
\begin{gather*}
u_t=\frac{1}{r_{k,l}}(kl\cosh lt\cos l\theta, kl\cosh lt\sin l\theta, kl\sinh kt\cos k\theta, kl\sinh kt\sin k\theta),\\
u_\theta=\frac{1}{r_{k,l}}(-kl\sinh lt\sin l\theta, kl\sinh lt\cos l\theta, -kl\cosh kt\sin k\theta, kl\cosh kt\cos k\theta),\\
u_{tt}=\frac{1}{r_{k,l}}(kl^2\sinh lt\cos l\theta,kl^2\sinh lt\sin l\theta, k^2l\cosh kt\cos k\theta, k^2l\cosh kt\sin k\theta)=\\=-u_{\theta\theta},\\
u_{t\theta}=\frac{1}{r_{k,l}}(-kl^2\cosh lt\sin l\theta, kl^2\cosh lt\cos l\theta, -k^2l\sinh kt\sin k\theta, k^2l\sinh kt\cos k\theta).
\end{gather*}

Notice that 
$$
e^{\omega}=|u_\theta|=|u_t|.
$$
Hence, $e^{\omega}=|u_t|=|u_\theta|=const$ along the boundary.
Also
$$
u_\theta \cdot u_t=0,
$$
which implies
\begin{gather*}
u_{\theta\theta}\cdot u_t=-u_\theta\cdot u_{t\theta},\\
u_{t\theta}\cdot u_t=-u_\theta\cdot u_{tt}=u_\theta\cdot u_{\theta \theta}.
\end{gather*}
One may also check that
$$
u_t\cdot u_{t\theta}=0.
$$
Using the projection formula we find that
\begin{gather*}
u_{tt}^\perp=u_{tt}-\frac{u_{tt}\cdot u_\theta}{|u_\theta|^2}u_\theta-\frac{u_{tt}\cdot u_t}{|u_t|^2}u_t,\\
u_{t\theta}^\perp=u_{t\theta}-\frac{u_{t\theta}\cdot u_\theta}{|u_\theta|^2}u_\theta-\frac{u_{t\theta}\cdot u_t}{|u_t|^2}u_t.
\end{gather*}
All together implies
\begin{gather*}
|u_{tt}^\perp|^2-|u_{t\theta}^\perp|^2=|u_{tt}|^2-|u_{t\theta}|^2.
\end{gather*}
The explicit computation yields
$$
|u_{tt}^\perp|^2-|u_{t\theta}^\perp|^2 =\frac{1}{r_{k,l}^2}(k^4l^2-k^2l^4)\neq 0.
$$
\end{example}

\begin{example}\label{ExM}
Consider the Fraser-Sargent annulus which corresponds to $k=2,l=1$ in the previous example. This surface is the orientable double cover of the critical M\"obius band $\mathbb M$ in $\mathbb B^4$. Thus, in order to show that $\mathcal H\neq 0$ on $\mathbb M$ it suffices to show that $\mathcal H\neq 0$ on its orientable cover. By the previous example one sees that $|u_{tt}^\perp|^2-|u_{t\theta}^\perp|^2=12\neq 0$.
\end{example}

\begin{example}\label{ExC}
Recall that the position vector of the critical $q-$catenoid $\mathbb K_q$ is given by
$$
u(t,\theta)=\frac{1}{r_q}(\cosh (qt)\cos(q\theta), \cosh (qt)\sin (q\theta),qt), 
$$
where $q\in\mathbb N$, $r_q=\sqrt{\cosh^2(qt_{1,0})+t^2_{1,0}}$ and $t_{1,0}$ is the unique positive solution of $\coth t=t$.
Then
\begin{gather*}
u_t=\frac{q}{r_q}(\sinh (qt)\cos(q\theta), \sinh(qt)\sin(q\theta), 1),\\
u_\theta=\frac{q}{r_q}(-\cosh(qt)\sin (q\theta), \cosh (qt)\cos (q\theta), 0),\\
u_{tt}=\frac{q^2}{r_q}(\cosh (qt)\cos (q\theta),\cosh(qt)\sin(q\theta), 0)=-u_{\theta\theta},\\
u_{t\theta}=\frac{q^2}{r_q}(-\sinh(qt)\sin(q\theta), \sinh(qt)\cos(q\theta), 0).
\end{gather*}
As in Example~\ref{ExFS} one finds that
\begin{gather*}
|u_\theta|=|u_t|,\\
u_\theta \cdot u_t=0,\\
u_{\theta\theta}\cdot u_t=-u_\theta\cdot u_{t\theta},\\
u_{t\theta}\cdot u_t=-u_\theta\cdot u_{tt}=u_\theta\cdot u_{\theta \theta},\\
u_t\cdot u_{t\theta}=0.
\end{gather*}
And that
\begin{gather*}
u_{tt}^\perp=u_{tt}-\frac{u_{tt}\cdot u_\theta}{|u_\theta|^2}u_\theta-\frac{u_{tt}\cdot u_t}{|u_t|^2}u_t,\\
u_{t\theta}^\perp=u_{t\theta}-\frac{u_{t\theta}\cdot u_\theta}{|u_\theta|^2}u_\theta-\frac{u_{t\theta}\cdot u_t}{|u_t|^2}u_t,\\
|u_{tt}^\perp|^2-|u_{t\theta}^\perp|^2=|u_{tt}|^2-|u_{t\theta}|^2.
\end{gather*}
The explicit computation yields
$$
|u_{tt}^\perp|^2-|u_{t\theta}^\perp|^2=\frac{q^4}{r_q^2}\neq 0.
$$
\end{example}

Now we pass to the proof of Theorem~\ref{main_intro}.

\begin{proof}[Proof of Theorem~\ref{main_intro}] 

It follows from Claim 6 that 
$$
L(\Omega_1)=8e^{-4\omega}\left(|\Omega_1|^2-|\Omega_2|^2\right)\Omega_1, \quad L(\Omega_2)=-8e^{-4\omega}\left(|\Omega_1|^2-|\Omega_2|^2\right)\Omega_2.
$$
It was shown in Examples~\ref{ExFS}-\ref{ExC} that $\mathcal H \neq 0$. Then without loss of generality we can assume that  $\Omega\cdot\Omega=1$. Proposition~\ref{dot} and Claim 6 imply that
$$
L(\Omega_1)=8e^{-4\omega}\Omega_1, \quad L(\Omega_2)=-8e^{-4\omega}\Omega_2
$$ 
and hence
\begin{gather*}
S(\Omega_1,\Omega_1)=-\int_{\Sigma} L(\Omega_1)\cdot\Omega_1 dv_g+\int_{\partial\Sigma}(\Omega_1\cdot\nabla^\perp_\eta \Omega_1-|\Omega_1|^2)ds_g=\\=-8\int_{\Sigma}e^{-4\omega}|\Omega_1|^2dv_g+\int_{\partial\Sigma}(\Omega_1\cdot\nabla^\perp_\eta \Omega_1-|\Omega_1|^2)ds_g. 
\end{gather*}
Let $X$ be the normalized position vector field in a neighbourhood of $\partial\Sigma$. Since $\Omega\cdot\Omega=1$ one has
$$
\nabla^\perp_X\Omega\cdot\Omega=0.
$$
Substituting $\Omega=\Omega_1+i\Omega_2$ we get
$$
\nabla_\eta^\perp\Omega_1\cdot\Omega_1+i\nabla^\perp_\eta\Omega_2\cdot\Omega_1=0.
$$ 
along the boundary. Whence
$$
\nabla^\perp_\eta \Omega_1\cdot \Omega_1=0
$$
along the boundary. Therefore, one has
\begin{gather*}
S(\Omega_1,\Omega_1)=-8\int_{\Sigma}e^{-4\omega}|\Omega_1|^2dv_g-\int_{\partial \Sigma}|\Omega_1|^2ds_g=\\=-8\int_{\Sigma}e^{-4\omega}|\Omega_1|^2dv_g-\Length(\partial \Sigma)<0.
\end{gather*}

It remains to prove that the fields $\Omega_1$ and $v_i^\perp,i=\overline{1,n}$ are linearly independent and $S$ is negative definite on $span\{\Omega_1, v_1^\perp,\ldots,v_n^\perp\}$. Consider $S(\Omega_1,v^\perp_i)$. Since $v_i^\perp$ are Jacobi fields (see Proposition~\ref{vperp}) then integrating by parts yields 
$$
S(\Omega_1,v^\perp_i)=\int_{\partial \Sigma} \Omega_1\cdot(\nabla^\perp_\eta v_i^\perp-v_i^\perp) ds_g.
$$
Let $\eta=e^{-\omega}\frac{\partial u}{\partial t}$ be the outward unit normal to $\partial\Sigma$ and $\tau=e^{-\omega}\frac{\partial u}{\partial \theta}$ be a unit tangent to $\partial \Sigma$. By the projection formula
$$
v^\perp_i=v_i-(v_i\cdot \tau)\tau-(v_i\cdot\eta)\eta.
$$
Differentiating yields
$$
\nabla_\eta^\perp v_i^\perp=-(v_i\cdot\eta)b_{11}.
$$
Here we have also used that $b_{12}=0$ along the boundary by Claim 5. Coming back to our computation and using Claim 5 we get
\begin{gather*}
S(\Omega_1,v^\perp_i)=\frac{1}{2}\int_{\partial \Sigma} e^{2\omega} b_{11}\cdot(- (v_i\cdot u) b_{11}-v_i^\perp) ds_g,
\end{gather*}
since $u=\eta$ along the boundary. Then
\begin{gather*}
S(\Omega_1,v^\perp_i)=\frac{1}{2}e^{2\omega}\int_{\partial \Sigma} (-( v_i\cdot u) |b_{11}|^2- b_{11}\cdot v_i)ds_g=-2e^{-2\omega}\int_{\partial \Sigma}u^i ds_g+\frac{1}{2}e^{2\omega}\int_{\partial \Sigma}b_{22}^ids_g,
\end{gather*}
where $u^i$ is the $i-$th coordinate of $\Sigma$ and $b_{11}^i$ is the $i-$th coordinate of the vector $b_{11}$. We have also used $b_{11}\cdot v_i^\perp=b_{11}\cdot v_i$ since $b_{11}\in \Gamma(N\Sigma)$, $|b_{11}|=2e^{-2\omega}|\Omega_1|=2e^{-2\omega}$ along the boundary and $b_{11}=-b_{22}$ by the minimality of $\Sigma$. Since $\Sigma$ is a free boundary minimal surface in $\mathbb B^n$ we get that $u^i$ is a Steklov eigenfunction with Steklov eigenvalue 1. This implies that
$$
\int_{\partial\Sigma}u^ids_g=0.
$$ 
We claim that 
$$
\int_{\partial\Sigma}b^i_{22}ds_g=0.
$$
 Observe that $b_{22}$ is the geodesic curvature of $\partial \Sigma$ in $\mathbb S^{n-1}=\partial \mathbb B^n$. Indeed, let $(\partial \Sigma)_i, i=\overline{1,2}$ be the $i-$th boundary component and $p\in (\partial \Sigma)_i$. Then one has
$$
B(\partial/\partial \theta,\partial/\partial \theta)(p)=\left(\frac{\partial^2 u}{\partial \theta^2}\right)^\perp(p) \in T_p\mathbb S^{n-1}
$$
thanks to the orthogonality condition. Therefore, $\left(\frac{\partial^2 u}{\partial \theta^2}\right)^\perp$ is the geodesic curvature of $(\partial \Sigma)_i$ in $\mathbb S^{n-1}$. Passing to the tangent vector $\tau$ of the unit length we get that $(\partial \Sigma)_j$ is parametrized naturally by a parameter $s\in [0,S]$ for some $S$. Let $w(s)$ be the velocity vector along $(\partial \Sigma)_j$. Then by definition
$$
b_{22}(p)=\frac{dw}{ds}(p),~\forall p\in (\partial \Sigma)_j.
$$
One has
$$
\int_0^{s'} b^i_{22}(s)ds=w^i(s')-w^i(0),~\forall s'\in [0,S].
$$
Notice that $w^i(0)=w^i(S)$ since $(\partial \Sigma)_j$ is closed. Then
$$
\int_{(\partial \Sigma)_j}b^i_{22}ds_g=\int_0^S b^i_{22}(s)ds=0.
$$
Hence,
$$
\int_{\partial \Sigma}b^i_{22}ds_g=0~\text{and}\quad S(\Omega_1,v^\perp_i)=0,
$$
which immediately implies that $\Omega_1$ and $v^\perp_i$ are linearly independent. Then by the linearity and the fact that $S(v_i^\perp,v_j^\perp)=0, \forall i\neq j$ (see Proposition~\ref{FS}) one gets that
$$
S(\Omega_1,v^\perp)=0,~\forall v\in \mathbb R^n.
$$
Hence, $\Omega_1$ and $v^\perp$ are linearly independent. 

For any vector field $X=\alpha \Omega_1+\beta v^\perp, \alpha^2+\beta^2\neq 0$, where $v\in \mathbb R^n$ one then has
\begin{gather*}
S(X,X)=\alpha^2S(\Omega_1,\Omega_1)-2\alpha\beta S(\Omega_1,v^\perp)+\beta^2S(v^\perp,v^\perp)=\\=\alpha^2S(\Omega_1,\Omega_1)+\beta^2S(v^\perp,v^\perp)<0.
\end{gather*}
\end{proof}

We finish this section with the following theorem

\begin{theorem}\label{indFS}
The index of Fraser-Sargent annuli in $\mathbb B^4$ is at least 6 and the nullity is at least 2.
\end{theorem}

\begin{proof}
Let $\Sigma$ be a Fraser-Sargent annulus. The normal bundle to $\Sigma$ is trivial since $\Sigma$ is orientable (see e.g.~\cite{fraser2007index}). Since $|\Omega_1|^2-|\Omega_2|^2=1$ and $\Omega_1\cdot\Omega_2=0$ then there exist global unit normal fields $N_1,N_2$ and a function $\mu$ such that $\Omega_1=\cosh\mu N_1, \Omega_2=\sinh\mu N_2$. Indeed, since $|\Omega_1| \geqslant 1$ we can set $N_1=\frac{\Omega_1}{|\Omega_1|}$. Then the field $N_2$ is defined as a unit field such that the orthogonal frame $u_t,u_\theta, N_1,N_2$ is positive oriented at every point $p\in\Sigma$. Moreover, $\Omega_2$ vanishes only on the boundary. Indeed, by Claim 5 one has
$$
\Omega_2=\frac{1}{2}B(\partial/\partial t,\partial/\partial \theta)=\frac{1}{2}u_{t\theta}^\perp.
$$
Suppose that $\Omega_2(p)=0$ for some $p\in\Sigma$. Then $u_{t\theta}^\perp(p)=0$ which implies that $u_{t\theta}(p)\in T_p\Sigma$. However, it follows from Example~\ref{ExFS} that $u_{t\theta}\cdot u_t=0$, which implies that $u_{t\theta}=\alpha u_\theta$ for some $\alpha\in\mathbb R\setminus \{0\}$. Using the explicit formulae for $u_{t\theta}$ and $u_\theta$ (see Example~\ref{ExFS} once again) then yields
$$
\begin{cases}
l\cosh lt=\alpha \sinh lt,\\
k\sinh lt=\alpha \cosh kt,\\
\end{cases}
$$
for some $t$ corresponding to the point $p$. This implies that 
$$
l\coth lt=k\tanh kt.
$$
The unique positive solution to this equation is $t=t_{k,l}$, which corresponds to the boundary of $\Sigma$. Thus, $\Omega_2=0$ only on $\partial\Sigma$. For the function $\mu$ one then has $\sinh \mu=0$ on $\partial\Sigma$. As in the proof of~\cite[Theorem 3.1 (2)]{kusner2018index} we will introduce a complex structure $J$ on $N\Sigma$ in the following way.
$$
JN_1=N_2, \quad JN_2=-N_1.
$$
By the Newlander-Nirenberg Theorem this complex structure is integrable since its Nijenhuis tensor vanishes. Note that $\nabla J=J\nabla$ and hence $\nabla^\perp J=J\nabla^\perp$. Also $J\Omega_2=0$ along the boundary of $\Sigma$. Moreover, one can show that
\begin{gather}\label{N}
\nabla^\perp_zN_1=i\mu_zN_2.
\end{gather}
Indeed, since by Claim 3 $\nabla^\perp_{\bar z}\Omega=0$ one has
\begin{gather*}
0=\nabla^\perp_{\bar z}(\Omega_1+i\Omega_2)=\nabla^\perp_{\bar z}(\cosh \mu N_1)+i\nabla^\perp_{\bar z}(\sinh\mu N_2)=\\\mu_{\bar z}\sinh\mu N_1+\cosh\mu\nabla^\perp_{\bar z}N_1+i\mu_{\bar z}\cosh\mu N_2+i\sinh\mu\nabla^\perp_{\bar z}N_2.
\end{gather*}
However, $\nabla^\perp_{\bar z}N_1 \parallel N_2$ and $\nabla^\perp_{\bar z}N_2 \parallel N_1$ as it easily follows from $|N_1|=1$ and $|N_2|=1$ by taking the covariant derivative $\nabla^\perp_{\bar z}$. Hence the previous identity implies
$$
\nabla^\perp_{\bar z}N_1=-i\mu_{\bar z}N_2.
$$
Conjugating yields~\eqref{N}.

Observe that $L(J\Omega_1)=0$. Indeed,
$$
\Delta^\perp (J\Omega_1)=J\Delta^\perp (\Omega_1)=-8e^{-4\omega}|\Omega_2|^2J\Omega_1=-8e^{-4\omega}\sinh^2\mu\cosh\mu N_2
$$
and
\begin{gather*}
\mathcal B(J\Omega_1)=\sum_{i,j=1}^2(b_{ij}\cdot J\Omega_1) b_{ij}=8e^{-4\omega}\left((\Omega_1\cdot J\Omega_1)\Omega_1+(\Omega_2\cdot J\Omega_1)\Omega_2\right)=\\=8e^{-4\omega}\sinh^2\mu\cosh\mu N_2.
\end{gather*}
We have used Claims 5 and 6 in both computations. Similarly, one can show that $L(J\Omega_2)=0$. Thus,
$$
S(J\Omega_1,J\Omega_1)=\int_{\partial \Sigma}(J\Omega_1\cdot\nabla^\perp_\eta J\Omega_1 -|J\Omega_1|^2)ds_g,\quad S(J\Omega_2,J\Omega_2)=0.
$$
Note also that $J\Omega\cdot J\Omega=\Omega\cdot\Omega=1$ and $|J\Omega_1|^2=|\Omega_1|^2=1$ along $\partial \Sigma$. As in the proof of Theorem~\ref{main_intro} one concludes that $J\Omega_1\cdot\nabla^\perp_\eta J\Omega_1=0$. Thus, $S(J\Omega_1,J\Omega_1)<0$.


We claim that the fields $\Omega_1, J\Omega_1, v_1^\perp,\ldots, v_4^\perp$ are linearly independent. Indeed, suppose that 
$$
\alpha\Omega_1+\beta J\Omega_1+v^\perp=0
$$
for some $v\in\mathbb R^4$ and $\alpha,\beta\in \mathbb R$. Applying the operator $\nabla^\perp_z$ one then gets
$$
\alpha\nabla^\perp_z\Omega_1+\beta J\nabla^\perp_z\Omega_1+\nabla^\perp_zv^\perp=0.
$$
Simplifying and using~\eqref{N} we get
\begin{gather}\label{nabla_z}
\alpha\mu_z\sinh\mu N_1+i\alpha\mu_z\cosh\mu N_2+\beta\mu_z\sinh\mu N_2-i\beta\mu_z\cosh\mu N_1-\notag \\-2e^{-2\omega}(v\cdot u_{\bar z})\cosh\mu N_1-2ie^{-2\omega}(v\cdot u_{\bar z})\sinh\mu N_2=0.
\end{gather}
Here we have also used the projection formula
$$
v^\perp=v-\frac{v\cdot u_z}{|u_z|^2}u_{\bar z}-\frac{v\cdot u_{\bar z}}{|u_{\bar z}|^2}u_z
$$
in the computation of $\nabla^\perp_zv^\perp$.

Since the fields $N_1$ and $N_2$ are linearly independent then~\eqref{nabla_z} implies 
$$
\begin{cases}
\alpha\mu_z\sinh\mu-i\beta\mu_z\cosh\mu-2e^{-2\omega}(v\cdot u_{\bar z})\cosh\mu=0,\\
i\alpha\mu_z\cos\mu+\beta\mu_z\sinh\mu-2ie^{-2\omega}(v\cdot u_{\bar z})\sinh\mu=0.
\end{cases}
$$
The second identity restricted on $\partial\Sigma$ implies that either $\mu_z=0$ on $\partial\Sigma$ or $\alpha=0$. If $\mu_z=0$ on $\partial\Sigma$ then the first identity restricted on $\partial\Sigma$ implies that $v\cdot u_{\bar z}=0$ on $\partial\Sigma$. The latter yields that $v\cdot u_z=0$ on $\partial\Sigma$, i.e. $v\perp T_p\Sigma$ for any $p\in\partial\Sigma$. In other words, $v\cdot u_t=0$ and $v\cdot u_\theta=0$ along $\partial\Sigma$. This cannot hold for all points on $\partial\Sigma$ and $v\neq 0$. Explicitly, one can take $(t,\theta)$ from the set $\{(\pm t_{k,l},0)\}$ and check that the system
$$
\begin{cases}
v\cdot u_t=0,\\
v\cdot u_\theta=0
\end{cases}
$$
admits only the trivial solution $v=0$. But if $v=0$ then one gets $\alpha\Omega_1+\beta J\Omega_1=0$ which can happen if and only if $\alpha=\beta=0$.

Now let us consider the case when $\alpha=0$. This implies that $\beta J\Omega_1+v^\perp=0$. Since $J\Omega_1=\cosh\mu N_2$ we get that $v^\perp\cdot N_1=0$ which implies $v\cdot N_1=0$. Notice that
$$
\Omega_1=\frac{1}{2}B(\partial/\partial t,\partial/\partial t)=\frac{1}{2}u_{tt}^\perp.
$$
Hence $N_1=\frac{u_{tt}^\perp}{|u_{tt}^\perp|}$. Then $v\cdot N_1=0$ implies that $v\cdot u_{tt}^\perp=0$. Using the computations in~\ref{ExFS} we find that
$$
u_{tt}^\perp=u_{tt}-\frac{u_{tt}\cdot u_t}{|u_t|^2}u_t.
$$
Computing $v\cdot u_{tt}^\perp$ at the points $(t,\theta)$ from the set $\{(0,0),(t_{k,l},0),(t_{k,l},\frac{\pi}{2l}), (t_{k,l},\frac{\pi}{2k})\}$ we get that the identity $v\cdot u_{tt}^\perp=0$ holds if and only if the vector $v=(v_1,\ldots,v_4)$ is zero as the unique solution of the following system
\begin{gather*}
\begin{cases}
kv_3=0,\\
\big(l\sinh(lt_{k,l})-\frac{1}{2}\frac{l\sinh(2lt_{k,l})+k\sinh(2kt_{k,l})}{\cosh^2(lt_{k,l})+\sinh^2(kt_{k,l})}\cosh(lt_{k,l})\big)v_1=0,\\
\big(l\sinh(lt_{k,l})-\frac{1}{2}\frac{l\sinh(2lt_{k,l})+k\sinh(2kt_{k,l})}{\cosh^2(lt_{k,l})+\sinh^2(kt_{k,l})}\cosh(lt_{k,l})\big)v_2+\\+\sin(\frac{\pi}{2}\frac{k}{l})\big(k\cosh(kt_{k,l})-\frac{1}{2}\frac{l\sinh(2lt_{k,l})+k\sinh(2kt_{k,l})}{\cosh^2(lt_{k,l})+\sinh^2(kt_{k,l})}\sinh(kt_{k,l})\big)v_4=0,\\
\sin(\frac{\pi}{2}\frac{l}{k})\big(l\sinh(lt_{k,l})-\frac{1}{2}\frac{l\sinh(2lt_{k,l})+k\sinh(2kt_{k,l})}{\cosh^2(lt_{k,l})+\sinh^2(kt_{k,l})}\cosh(lt_{k,l})\big)v_2+\\+\big(k\cosh(kt_{k,l})-\frac{1}{2}\frac{l\sinh(2lt_{k,l})+k\sinh(2kt_{k,l})}{\cosh^2(lt_{k,l})+\sinh^2(kt_{k,l})}\sinh(kt_{k,l})\big)v_4=0.
\end{cases}
\end{gather*}
Notice that $l\sinh(lt_{k,l})-\frac{1}{2}\frac{l\sinh(2lt_{k,l})+k\sinh(2kt_{k,l})}{\cosh^2(lt_{k,l})+\sinh^2(kt_{k,l})}\cosh(lt_{k,l}) \neq 0$ and $k\cosh(kt_{k,l})-\frac{1}{2}\frac{l\sinh(2lt_{k,l})+k\sinh(2kt_{k,l})}{\cosh^2(lt_{k,l})+\sinh^2(kt_{k,l})}\sinh(kt_{k,l})\neq 0$ because otherwise we get that $l\tanh(lt_{k,l})=k\coth(kt_{k,l})$ which is impossible since $l\coth(lt_{k,l})=k\tanh(kt_{k,l})$ and $k>l$. This implies that $\beta$ is also $0$. We have shown that the fields $\Omega_1, J\Omega_1, v_1^\perp,\ldots, v_4^\perp$ are linearly independent and the first statement of the theorem is proved. 


Now let us prove that $\Nul(\Sigma)\geqslant 2$. As we have already seen $S(J\Omega_2,J\Omega_2)=0$. Moreover, it is known that the field $u^\perp$ is a Jacobi field vanishing on $\partial\Sigma$. Hence, $S(u^\perp,u^\perp)=0$. Our aim is to show that the fields $J\Omega_2$ and $u^\perp$ are linearly independent. Assume the contrary, i.e. there exists a real number $\alpha\neq 0$ such that $u^\perp=\alpha J\Omega_2$. One can see that
$$
N_1=\frac{u_{tt}^\perp}{|u_{tt}^\perp|}~\text{and}~N_2=\frac{u_{t\theta}^\perp}{|u_{t\theta}^\perp|}
$$
inside $\Sigma$. Hence
$$
\cosh\mu=\frac{1}{2}|u_{tt}^\perp|~\text{and}~\sinh\mu=\frac{1}{2}|u_{t\theta}^\perp|.
$$
Therefore,
$$
J\Omega_2=-\sinh\mu N_1=-\frac{1}{2}\frac{|u_{t\theta}^\perp|}{|u_{tt}^\perp|}u_{tt}^\perp
$$
and the assumption $u^\perp=\alpha J\Omega_2$ implies that
\begin{gather}\label{aim}
u^\perp\cdot u_{tt}^\perp=-\frac{\alpha}{2}|u_{t\theta}^\perp||u_{tt}^\perp|.
\end{gather}
An explicit computation yields that
\begin{gather*}
u^\perp\cdot u_{tt}^\perp=\frac{kl}{r_{k,l}^2}(kl(\sinh^2lt+\cosh^2kt)-\frac{A}{2}(l\sinh2lt+k\sinh2kt)-\\-\frac{B}{2}(k\sinh2lt+l\sinh2kt)+AB(\cosh^2lt+\sinh^2kt)),\\
|u_{tt}^\perp|=\frac{kl}{r_{k,l}^2}\sqrt{l^2\sinh^2lt+k^2\cosh^2kt-B(l\sinh2lt+k\sinh2kt)+B^2(\cosh^2lt+\sinh^2kt)},\\
|u_{t\theta}^\perp|=\frac{kl}{r_{k,l}^2}\sqrt{l^2\cosh^2lt+k^2\sinh^2kt-B(l\sinh2lt+k\sinh2kt)+B^2(\cosh^2lt+\sinh^2kt)},
\end{gather*}
where 
$$
A=\frac{k\sinh2lt+l\sinh2kt}{2(\cosh^2lt+\sinh^2kt)}~\text{and}~B=\frac{l\sinh2lt+k\sinh2kt}{2(\cosh^2lt+\sinh^2kt)}.
$$
Plugging the above expressions into~\eqref{aim} and performing a tedious computation yield that $\alpha$ cannot be constant whenever $k\neq l$. We arrive at a contradiction.
\end{proof}

\medskip

\section{Proof of Theorem~\ref{indexES}}\label{ES}

 In this section we prove Theorem~\ref{indexES} and deduce some corollaries.

\begin{proof}[Proof of Theorem~\ref{indexES}] The proof is a straightforward adaptation of \cite[Theorem 3.3]{karpukhin2021index} for the Steklov setting. For the sake of completeness we give it here.

Let $V$ be the maximal negative space of the form $S_S$, i.e. $\dim V=\Ind_S(\Sigma)$. Suppose that 
$$
n \Ind_S(\Sigma) <\Ind_E(\Sigma).
$$
Then there exists a harmonic vector field $X$ such that $S_E(X,X)<0$ but the components $X^i, i=\overline{1,n}$ of $X$ are perpendicular to any function $f\in V$. Then $S_S(X^i,X^i) \geqslant 0, i=\overline{1,n}$. However, one can see that
$$
\sum_{i=1}^nS_S(X^i,X^i)=S_E(X,X) \geqslant 0.
$$ 
We arrive at a contradiction.
\end{proof}

Notice that in the previous theorem the orientability assumption is not needed. Combining Theorem~\ref{indexES} with Theorem~\ref{lima} one gets the following corollary

\begin{corollary}\label{cor}
Let $\Sigma$ be a (orientable or non-orientable) free boundary minimal surface in $\mathbb B^n$. Then
$$
\Ind(\Sigma) \leqslant n\Ind_S(\Sigma)+\dim \mathcal M(\Sigma).
$$
\end{corollary}

One can also extract the following corollary which could be of independent interest

\begin{corollary}\label{n_and_3}
Let $\Sigma$ be a free boundary minimal surface in $\mathbb B^n$ different from the plane disk. Then its spectral index satisfies
$$
n\Ind_S(\Sigma)+\dim \mathcal M(\Sigma) \geqslant n.
$$
Moreover, if $n=3$ then
$$
3\Ind_S(\Sigma)+\dim \mathcal M(\Sigma) \geqslant 4.
$$
\end{corollary}

\begin{proof}
The corollary immediately follows from Theorems~\ref{FS} and Corollary~\ref{cor}. If $n=3$ then we use Theorem~\ref{devyver} in place of Theorem~~\ref{FS}.
\end{proof}

We finish this section with the proof of Theorem~\ref{+n}.

\begin{proof}[Proof of Theorem~\ref{+n}]
Since $\Sigma$ is a free boundary minimal hypersurface in $\mathbb B^n$ then the coordinate functions $u_1,\ldots,u_n$ are Steklov eigenfunctions with eigenvalue 1. Note that  $u_1,\ldots,u_n$ are linearly independent as soon as $\Sigma$ is not flat. Suppose that $\Ind_S(\Sigma)=k$, i.e. there are $k$ linearly independent Steklov eigenfunctions $\varphi_1,\ldots,\varphi_k$ with eigenvalues $\sigma_i<1, i=\overline{1,k}$ respectively. Without loss of generality one can assume that $\varphi_1,\ldots,\varphi_k$ are orthonormal with respect to the $L^2(\partial\Sigma)-$norm. Consider $V=span\{\varphi_1,\ldots,\varphi_k,u_1,\ldots,u_n\}$. One can see that $\dim V=k+n$. We claim that the index form $S$ is negative definite on $V$. Indeed, let $\psi\in V$, i.e. $\psi=\sum_{i=1}^k\alpha_i\varphi_i+\sum_{j=1}^n\beta_ju_j$. Since $\Sigma$ is a hypersurface then the index form $S$ on $\psi$ reads:
\begin{gather}\label{form}
S(\psi,\psi)=-\int_\Sigma (\Delta_g\psi+|B|^2\psi)\psi dv_g+\int_{\partial\Sigma}\left(\frac{\partial \psi}{\partial \eta}-\psi\right)\psi ds_g.
\end{gather}
Obviously, $\Delta_g\psi=0$, since it's a linear combination of Steklov eigenfunctions. Moreover,  
$$
\frac{\partial \psi}{\partial \eta}=\sum_{i=1}^k\alpha_i\sigma_i\varphi_i+\sum_{j=1}^n\beta_ju_j~\text{on $\partial\Sigma$}.
$$
One may easily check that
\begin{gather}\label{form1}
\int_{\partial\Sigma}\frac{\partial \psi}{\partial \eta}\psi ds_g=\Length(\partial\Sigma)\sum_{i=1}^k\alpha^2_i\sigma_i+\int_{\partial\Sigma}\left(\sum_{j=1}^n\beta_ju_j\right)^2ds_g.
\end{gather}
Similarly,
\begin{gather}\label{form2}
\int_{\partial\Sigma}\psi^2 ds_g=\Length(\partial\Sigma)\sum_{i=1}^k\alpha^2_i+\int_{\partial\Sigma}\left(\sum_{j=1}^n\beta_ju_j\right)^2ds_g.
\end{gather}
Plugging~\ref{form1} and~\ref{form2} into~\ref{form} one gets that $S(\psi,\psi)<0$ as soon as $\Sigma$ is not flat since $\sigma_i<1,i=\overline{1,k}$. Therefore,
$$
\Ind(\Sigma)\geqslant k+n=\Ind_S(\Sigma)+n.
$$
\end{proof}

\section{Proof of Theorem~\ref{mobius_intro}} \label{mobius}

Our strategy is as follows. We pass to the orientable cover of $\mathbb M$ which correspond to the Fraser-Sargent surface with $k=2,l=1$. Let's denote this cover by $\widetilde{\mathbb M}$. Then by Theorem~\ref{main_intro} the fields $\Omega_1$ and $v_i^\perp, i=\overline{1,4}$ contribute to the index of  $\widetilde{\mathbb M}$. We need to show that the fields $\Omega_1$ and $v_i^\perp, i=\overline{1,4}$ descend to $\mathbb M$. This will imply that $\Ind(\mathbb M) \geqslant 5$. In order to get the inverse inequality we will then apply Corollary~\ref{cor}.

Recall that the position vector of $\mathbb M$ is given by
$$
u(t,\theta)=(2\sinh t\cos\theta, 2\sinh t\sin \theta, \cosh 2t\cos 2\theta, \cosh 2t\sin 2\theta)
$$
and 
\begin{gather*}
u_t=(2\cosh t\cos\theta, 2\cosh t\sin \theta, 2\sinh 2t\cos 2\theta, 2\sinh 2t\sin 2\theta),\\
u_\theta=(-2\sinh t\sin \theta, 2\sinh t\cos \theta, -2\cosh 2t\sin 2\theta, 2\cosh 2t\cos 2\theta),\\
u_{tt}=(2\sinh t\cos \theta,2\sinh t\sin \theta, 4\cosh 2t\cos2\theta, 4\cosh2t\sin 2\theta)=-u_{\theta\theta},\\
u_{t\theta}=(-2\cosh t\sin\theta, 2\cosh t\cos\theta, -4\sinh 2t\sin 2\theta, 4\sinh 2t\cos 2\theta),\\
u_{tt}^\perp=u_{tt}-\frac{u_{tt}\cdot u_\theta}{|u_\theta|^2}u_\theta-\frac{u_{tt}\cdot u_t}{|u_t|^2}u_t.
\end{gather*}

One may easily check that
\begin{gather*}
u_\theta(t,\theta)=u_\theta(-t,\theta+\pi),\\
u_t(t,\theta)=-u_t(-t,\theta+\pi),\\
u_{tt}(t,\theta)=u_{tt}(-t,\theta+\pi),\\
u_{t\theta}(t,\theta)=-u_{t\theta}(-t,\theta+\pi).\\
\end{gather*}
We are interested in vector fields $X$ that satisfy the condition $X(t,\theta)=X(-t,\theta+\pi)$. Obviously, a tangent vector field $X=au_t+bu_\theta$ satisfies the condition $X(t,\theta)=X(-t,\theta+\pi)$ if and only if
\begin{gather*}
a(-t,\theta+\pi)=-a(t,\theta),\\
b(-t,\theta+\pi)=b(t,\theta).\\
\end{gather*}
A straightforward computation yields that
 \begin{gather*}
u_{tt}^\perp(-t,\theta+\pi)=u_{tt}^\perp(t,\theta).
\end{gather*}
Therefore, the field $u_{tt}^\perp$ descends to a field on $\mathbb M$. Hence, $\Omega_1$ descends to $\mathbb M$.

Observing that the fields $v_i^\perp, i=\overline{1,4}$ also satisfy the condition $X(t,\theta)=X(-t,\theta+\pi)$ we conclude that $\Ind(\mathbb M) \geqslant 5$. 

In order to get the inverse inequality we observe that since $\mathbb M$ is given by first Steklov eigenfunctions then $\Ind_S(\mathbb M)=1$. Furthermore, the moduli space of conformal structures on the M\"obius band is isomorphic to the ray $\mathbb R_{>0}$. Hence, $\dim\mathcal M(\mathbb M)=1$ and by Corollary~\ref{cor} one has $\Ind(\mathbb M) \leqslant 5$. Thus, $\Ind(\mathbb M)=5$.

\subsection{Another proof of Theorem~\ref{cat}} \label{Cat}

 By Theorem~\ref{main_intro} (see also Theorem~\ref{devyver}) the index of $\mathbb K$ is at least 4. Since $\mathbb K$ is given by first Steklov eigenfunctions then we get that $\Ind_S(\mathbb K)=1$.  The moduli space of conformal structures on the annulus is isomorphic to the ray $\mathbb R_{>0}$ hence $\dim\mathcal M(\mathbb K)=1$. Then Corollary~\ref{cor} implies that $\Ind(\mathbb K) \leqslant 4$. Thus, $\Ind(\mathbb K)=4$.

 \section{Appendix}\label{appendix}
 In this section we prove Theorem~\ref{lima}. The proof follows from the same steps as the proofs of Propositions 6.5 and 7.3 in~\cite{fraser2016sharp}.
 
 Let us recall that a vector field $Y$ on $\mathbb B^n$ is said to be \textit{conformal} if for any local orthonormal basis $\{e_1,e_2\}$ in $\Gamma(T\Sigma)$ one has
 $$
 \nabla_{e_1}Y\cdot e_2=-\nabla_{e_2}Y\cdot e_1,\quad  \nabla_{e_1}Y\cdot e_1=\nabla_{e_2}Y\cdot e_2.
 $$
 
 \begin{remark}
 If $Y$ is a conformal vector field on $\mathbb B^n$ then for any tangent vector field $X$ on $\Sigma$ one has that $\nabla_XY\cdot X=const$.
 \end{remark}
 
 The following lemma reveals the importance of conformal vector fields
 
 \begin{lemma}[Fraser-Schoen~\cite{fraser2016sharp}]\label{FSapp}
 If $Y$ is a conformal vector field then for the quadratic forms of the second variations of the energy and volume functionals one has
 $$
 S_E(Y,Y)=S(Y^\perp,Y^\perp).
 $$
 \end{lemma}
  
 \begin{proof}[Proof of Theorem~\ref{lima}] We will provide a proof for the case of non-orientable free boundary minimal surfaces. The proof for the case of orientable free boundary minimal surfaces is easier and follows from the same steps. 
 
 Let $(x,y)$ be isothermal coordinates on $\Sigma$ such that $\partial_y$ is tangent to $\partial\Sigma$ and $z=x+iy$ be the corresponding complex coordinate. As before $\Sigma$ is given by the immersion $u\colon \Sigma \to \mathbb B^n$. Let $V$ be the maximal negative space of the form $S$, i.e. $\dim V=\Ind(\Sigma)$. Consider $\xi\in V$ and $X\in\Gamma(T\Sigma)$. If we will pass to the orientable cover $\widetilde{\Sigma}$ of $\Sigma$ then the fields $\xi$ and $X$ lift to the vector fields $\tilde\xi \in \Gamma(N\Sigma)$ and $\tilde X\in\Gamma(T\Sigma)$ respectively which are invariant under the involution $\iota$ changing the orientation. In this case $\widetilde\Sigma$ is given by the $\iota-$ invariant immersion $\tilde u\colon \widetilde\Sigma\to\mathbb B^n$.
 
 Consider the vector field $\tilde Y=\tilde X+\tilde\xi$. Let's suppose that this field is conformal for some $\tilde \xi$ which form a vector space $U\subset V$. Then by Lemma~\ref{FSapp} one has
 $$
 S_E(\tilde Y,\tilde Y)=S(\tilde\xi,\tilde\xi).
 $$
 The latter would imply that the fields $\tilde Y$ and $\tilde\xi$ descend to the fields $Y$ and $\xi$ on $\Sigma$ with the property
  $$
  S_E( Y, Y)=S(\xi,\xi).
 $$
 Therefore, one would get that
 $$
 \Ind_E(\Sigma) \geqslant \dim U.
 $$
 We will show that $\dim U \geqslant \dim V-\dim \mathcal M(\Sigma)$. In other words, for at least $\dim \mathcal M(\Sigma)$-codimensional subspace of $V$ one can find a tangent vector field $X$ such that the field $Y=X+\xi$ is conformal.
 
 The condition that the vector field $\tilde Y$ is conformal reads as
$$
 \nabla_{\partial_x}\tilde Y\cdot \tilde u_y=-\nabla_{\partial_y}\tilde Y\cdot \tilde u_x,\quad  \nabla_{\partial_x}\tilde Y\cdot \tilde u_x=\nabla_{\partial_y}\tilde Y\cdot \tilde u_y.
$$
 In terms of the complex coordinate $z$ the previous equations become
 \begin{gather}\label{conformal}
  \nabla_{ z}\tilde Y\cdot \tilde u_{ z}=0.
 \end{gather} 
 Here we have also used that the field $\tilde Y$ is real.
 
 Now substitute $\tilde Y=\tilde X+\tilde\xi$ into \eqref{conformal}. Simplifying, we get 
 \begin{gather}\label{conformal10}
D^{1,0}\tilde X^{0,1}=-(\nabla^{1,0} \tilde\xi)^\top,
  \end{gather} 
where $D^{1,0}=\nabla^\top_{z}\otimes d{ z}, \nabla^{1,0}=\nabla_{z}\otimes d{ z}$, $\tilde X^{1,0}$ and $\tilde X^{0,1}$ are the components of $\tilde X$ expressed in the complex coordinate $z$ such that $\tilde X^{1,0} \in span\{\tilde u_z\}$ and $\tilde X^{0,1} \in span\{\tilde u_{\bar z}\}$ i.e. $\tilde X=\tilde X^{1,0}+\tilde X^{0,1}$. In order to get formula \eqref{conformal10} we have also used Claim 1 which yields
$$
\nabla_{z}\tilde X^{1,0}\cdot \tilde u_{z}=0
$$ 
and
$$
\nabla_{z}\tilde\xi\cdot \tilde u_{\bar z}=-\tilde\xi\cdot \tilde u_{z\bar z}=0.
$$ 
Notice also that the field $X$ has to be admissible, i.e. tangent to $\partial \mathbb B^n$. This yields that $X=\varphi u_y$ along $\partial \Sigma$. In terms of the complex coordinate $z$ one gets $\re X^{0,1}=0$ along $\partial \Sigma$. Therefore, we need to study the solvability of the problem
\begin{gather}\label{problem}
\begin{cases}
D^{1,0}\tilde X^{0,1}=-(\nabla^{1,0} \tilde\xi)^\top~\text{in $\widetilde\Sigma$},\\
\re \tilde X^{0,1}=0~\text{on $\partial\widetilde\Sigma$},\\
\iota_*\tilde X^{0,1}=\tilde X^{0,1}.
\end{cases}
\end{gather}
Consider the operator $D^{1,0}\colon \Gamma_{\im,\iota}(T^{0,1}\widetilde\Sigma) \to\Gamma_\iota  (T^{0,1}\widetilde\Sigma\otimes \Lambda^{1,0}\widetilde\Sigma)$, where $\iota$ in the subscript denotes the $\iota-$ invariant sections and $\im$ in the subscript denotes sections which are pure imaginary on $\partial\widetilde\Sigma$. By the Fredholm alternative problem~\eqref{problem} is solvable if and only if $(\nabla^{1,0} \tilde\xi)^\top$ is $L^2-$orthogonal to $\Ker (D^{1,0})^*$, where $(D^{1,0})^*$ is the $L^2-$adjoint operator to $D^{1,0}$. The integration by parts yields that 
$$
(D^{0,1})^*\colon \Gamma_{\re,\iota}  (T^{0,1}\widetilde\Sigma\otimes \Lambda^{1,0}\widetilde\Sigma) \to \Gamma_\iota(T^{0,1}\widetilde\Sigma).
$$
Here $\Gamma_{\re,\iota}  (T^{0,1}\widetilde\Sigma\otimes \Lambda^{1,0}\widetilde\Sigma)$ denotes the $\iota-$invariant sections of $T^{0,1}\widetilde\Sigma\otimes \Lambda^{1,0}\widetilde\Sigma$ which are pure real on $\partial\widetilde\Sigma$. Further, by the computation on page 227 of~\cite{wells1980differential} one has $(D^{1,0})^*=-\ast\bar\partial\ast$, where $\ast$ is the Hodge star operator. Moreover, it is easy to see that $\ast\omega=-i\omega,~\forall \omega\in \Gamma(T^{0,1}\widetilde\Sigma\otimes \Lambda^{1,0}\widetilde\Sigma)$. Hence, $\Ker(D^{0,1})^*=H^0_{\iota, \re}(T^{0,1}\widetilde\Sigma\otimes \Lambda^{1,0}\widetilde\Sigma)$ that is the space of $\iota-$invariant holomorphic sections of $T^{0,1}\widetilde\Sigma\otimes \Lambda^{1,0}\widetilde\Sigma$ which are pure real on $\partial\widetilde\Sigma$. Using the Hermitian metric the bundle $T^{0,1}\widetilde\Sigma\otimes \Lambda^{1,0}\widetilde\Sigma$ can be identified with the bundle $(T^{1,0}\widetilde\Sigma)^*\otimes \Lambda^{1,0}\widetilde\Sigma$ and the space $H^0_{\iota, \re}(T^{0,1}\widetilde\Sigma\otimes \Lambda^{1,0}\widetilde\Sigma)$ can be identified with the \textit{space of $\iota-$invariant holomorphic quadratic differentials taking real values on $\partial\widetilde\Sigma$}. Then $\dim H^0_{\iota, \re}(T^{0,1}\widetilde\Sigma\otimes \Lambda^{1,0}\widetilde\Sigma)=\dim \mathcal M(\Sigma)$ (see for instance~\cite[Section 2]{jost1993minimal}). Hence, problem~\eqref{problem} is solvable if and only if 

 $$
((\nabla^{0,1} \tilde\xi)^\top,W)_{L^2}=0,~\forall W\in H^0_{\iota, \re}(T^{1,0}\widetilde\Sigma\otimes \Lambda^{0,1}\widetilde\Sigma).
$$
 Then we take the space 
 $$
 \{\xi\in V~|~((\nabla^{0,1} \tilde\xi)^\top,W)_{L^2}=0,~\forall W\in H^0_{\iota, \re}(T^{1,0}\widetilde\Sigma\otimes \Lambda^{0,1}\widetilde\Sigma)\}
 $$ 
 as the desired space $U$. Clearly, $\dim U \geqslant \dim V-\dim H^0_{\iota, \re}(T^{1,0}\widetilde\Sigma\otimes \Lambda^{0,1}\widetilde\Sigma) $. As as a result one gets
 $$
 \Ind_E(\Sigma) \geqslant \dim U\geqslant \dim V-\dim \mathcal M(\Sigma)=\Ind(\Sigma)-\dim \mathcal M(\Sigma).
 $$

 \end{proof}

\bibliography{mybib}
\bibliographystyle{alpha}

\end{document}